\documentclass[10pt]{amsart}
\usepackage{tikz,array, verbatim}
\usepackage{amsfonts, amsmath, latexsym, epsfig, caption}
\usepackage{amssymb, color}

\usepackage{epsf}
\usepackage{array}
\usepackage{ragged2e}
\usepackage{hyperref}
\newcolumntype{C}[1]{>{\centering\let\newline\\\arraybackslash\hspace{0pt}}m{#1}}
\newcolumntype{L}[1]{>{\raggedright\let\newline\\\arraybackslash\hspace{0pt}}m{#1}}
\newcolumntype{R}[1]{>{\raggedleft\let\newline\\\arraybackslash\hspace{0pt}}m{#1}}

\newif\ifshowvc
\showvctrue
\showvcfalse  

\ifshowvc
\input{vc}
\fi

\newcommand{\ip}[1]{\left \langle #1 \right \rangle}
\DeclareMathOperator{\Span}{span}
\DeclareMathOperator{\Syl}{Syl}
\DeclareMathOperator{\Vor}{Vor}
\DeclareMathOperator{\Res}{Res}
\DeclareMathOperator{\St}{St}

\newcommand{\bG}{\ensuremath{\mathbf{G}}}
\newcommand{\hp}{\ensuremath{\mathfrak{H}}}
\newcommand{\cH}{\ensuremath{\mathcal{H}}}
\newcommand{\cS}{\ensuremath{\mathcal{S}}}
\newcommand{\set}[1]{\left\{#1\right\}}
\newcommand{\OO}{\ensuremath{\mathcal{O}}}
\newcommand{\RR}{\ensuremath{\mathbb{R}}}

\newcommand{\QQ}{\ensuremath{\mathbb{Q}}}
\newcommand{\CC}{\ensuremath{\mathbb{C}}}
\newcommand{\ZZ}{\ensuremath{\mathbb{Z}}}

\theoremstyle{plain}
\newtheorem{theorem}{Theorem}[section]
\newtheorem{proposition}[theorem]{Proposition}
\newtheorem{corollary}[theorem]{Corollary}
\newtheorem{lemma}[theorem]{Lemma}

\theoremstyle{definition}
\newtheorem{definition}[theorem]{Definition}

\theoremstyle{remark}
\newtheorem{remark}[theorem]{Remark}

\DeclareMathOperator{\Aut}{Aut}

\DeclareMathOperator{\SL}{SL}

\DeclareMathOperator{\GL}{GL}
\DeclareMathOperator{\PSL}{PSL}

\DeclareMathOperator{\Stab}{Stab}
\DeclareMathOperator{\rank}{rank}

\DeclareMathOperator{\Tr}{Tr}
\DeclareMathOperator{\cd}{cd}
\DeclareMathOperator{\vcd}{vcd}
\DeclareMathOperator{\Dim}{dim}

\newcommand{\theoremcite}[1]{\emph{\cite{#1}}}
\newcommand{\rftheoremcite}[2]{\emph{\cite[#1]{#2}}}

\def\QuotS#1#2{\leavevmode\kern-.0em\raise.2ex\hbox{$#1$}\kern-.1em/\kern-.1em\lower.25ex\hbox{$#2$}}

\newcommand{\fS}{\mathfrak{S}}

\begin{document}
\author[H. Gangl]{Herbert Gangl}
\address{H. Gangl, Department of Mathematical Sciences, South Road, Durham DH1 3LE, United Kingdom}
\email{herbert.gangl@durham.ac.uk}
\urladdr{\url{http://maths.dur.ac.uk/~dma0hg/}}

\author[P. E. Gunnells]{Paul E. Gunnells}
\address{P. E. Gunnells, Department of Mathematics and Statistics, LGRT 1115L, University of Massachusetts, Amherst, MA 01003, USA}
\email{gunnells@math.umass.edu}
\urladdr{\url{https://www.math.umass.edu/~gunnells/}}

\author[J. Hanke]{Jonathan Hanke}
\address{J. Hanke, One Palmer Square, Suite 441, Princeton, NJ 08542, USA}
\email{jonhanke@gmail.com}
\urladdr{\url{http://www.jonhanke.com}}

\author[A. Sch\"urmann]{Achill Sch\"urmann}
\address{A. Sch\"urmann, Universit\"at Rostock, Institute of Mathematics, 18051 Rostock, Germany}
\email{achill.schuermann@uni-rostock.de}
\urladdr{\url{http://www.geometrie.uni-rostock.de/}}

\author[M. D. Sikiri\'c]{Mathieu Dutour Sikiri\'c}
\address{M. D. Sikiri\'c, Rudjer Boskovi\'c Institute, Bijenicka 54, 10000 Zagreb, Croatia}
\email{mathieu.dutour@gmail.com}
\urladdr{\url{http://drobilica.irb.hr/~mathieu/}}

\author[D. Yasaki]{Dan Yasaki}
\address{D. Yasaki, Department of Mathematics and Statistics, University of North Carolina at Greensboro, Greensboro, NC 27412, USA}
\email{d\_yasaki@uncg.edu}
\urladdr{\url{http://www.uncg.edu/~d_yasaki/}}

\thanks{MDS was partially supported by the Croatian Ministry of
Science, Education and Sport under contract 098-0982705-2707 and by
the Humboldt Foundation.  PG was partially supported by the NSF under
contract DMS 1101640.  JH was partially supported by the NSF under contract DMS-0603976.
The authors thank the American Institute of
Mathematics, where this research was initiated.}



\title{On the cohomology of linear groups over imaginary quadratic fields}
\date{25 November, 2013}

\begin{abstract}
Let $\Gamma$ be the group $\GL_N (\OO_D)$, 
where $\OO_D$ is the ring of integers in the imaginary quadratic field with discriminant $D<0$. 
In this paper we investigate the cohomology of
$\Gamma$ for $N=3,4$ and for a selection of discriminants: $D\geq -24$
when $N=3$, and $D=-3,-4$ when $N=4$.  In particular we compute the
integral cohomology of $\Gamma$ up to $p$-power torsion for small
primes $p$.  Our main tool is the polyhedral reduction theory for
$\Gamma$ developed by Ash \cite[Ch.~II]{AMRT_SmoothCompactification}
and Koecher \cite{koecher}.  Our results extend work of Staffeldt
\cite{k3gauss}, who treated the case $n=3$, $D=-4$.  In a sequel
\cite{aim-ktheory} to this paper, we will apply some of these results
to the computations with the $K$-groups $K_{4} (\OO_{D})$, when $D=-3,-4$.

\end{abstract}

\maketitle
\ifshowvc
\let\thefootnote\relax
\footnotetext{Base revision~\GITAbrHash, \GITAuthorDate,
\GITAuthorName.}
\fi

\section{Introduction}

\subsection{} Let $F$ be an imaginary quadratic field of discriminant
$D<0$, let $\OO = \OO_{D}$ be its ring of integers, and let $\Gamma$ be
the group $\GL_{N} (\OO)$.  The homology and cohomology of $\Gamma$
when $N=2$ --- or rather 
its
close cousin the Bianchi group
$\PSL_{2} (\OO)$ --- have been well studied in the literature.  For an
(incomplete) selection of results we refer to
\cite{rahm1,rahm2,berkove,schwermer-vogtmann,vogtmann,sengun,cremona}.
Today we have a good understanding of a wide range of examples, and
one can even compute them for very large discriminants
(cf.~\cite{yasaki}).  For $N>2$, on the other hand, the group $\Gamma$
has not received the same attention.  The first example known to us is
the work of Staffeldt \cite{k3gauss}.  He treated the case $N=3$,
$D=-4$ with the goal of understanding the $3$-torsion in $K_{3}
(\ZZ[\sqrt{-1}])$.  The second example is \cite{coul-herm}, which
investigates the case of the groups $\GL(L)$ for $L$ an $\OO$-lattice
that is not necessarily a free $\OO$-module.  This allows the authors
to compute the Hermite constants of those rings in case $D\geq -10$
and $\rank(L) \leq 3$. Our methods apply as well to the non-free case, and
the corresponding cohomology computations would be useful when
investigating automorphic forms over number fields that are not
principal ideal domains (cf.~\cite[Appendix]{SteinBook} for more about
the connection between cohomology of arithmetic groups and automorphic
forms).  

In this paper we rectify this situation somewhat by beginning the
first systematic computations for higher rank linear groups over
$\OO$.  In particular investigate the cohomology of $\Gamma$ for
$N=3,4$ and for a selection of discriminants: $D\geq -24$ when $N=3$,
and $D=-3,-4$ when $N=4$.  We explicitly compute the polyhedral
reduction domains arising from Voronoi's theory of perfect forms, as
generalized by Ash \cite[Ch.~II]{AMRT_SmoothCompactification} and
Koecher \cite{koecher}.  This allows us to compute the integral
cohomology of $\Gamma$ up to $p$-power torsion for small primes $p$.
In a sequel \cite{aim-ktheory} to this paper, we will apply some of
these results to computations with the $K$-groups $K_{4} (\OO_{D})$,
when $D=-3,-4$.

\subsection{} Here is a guide to the paper (which closely follows the
structure of the first five sections of \cite{PerfFormModGrp}).  In
Section~\ref{sec:polyhedralcone} we recall the explicit reduction
theory we need to build our chain complexes to compute cohomology.  In
Section \ref{sec:cohomologyandhomology} we define the complexes, and
explain the relation between what we compute and the cohomology of
$\Gamma$.  In Section \ref{sec:mass} we describe a ``mass formula''
for the cells in our tessellations that provides a strong 
computational check on the correctness of our constructions.  In
Section~\ref{sec:explicit} we give an explicit representative for the
nontrivial class in the top cohomological degree; this construction is
motivated by a similar construction in \cite{PerfFormModGrp,
PEV-HG-CS-announce}.  Finally, in Section~\ref{sec:tables} we give the
results of our computations.

\subsection{Acknowledgments} We thank A.~Ash, P.~Elbaz-Vincent, and
C.~Soul\'e for helpful discussions.  This research, and the research
in the companion papers \cite{aim-voronoi, aim-ktheory} was conducted
as part of a ``SQuaRE'' (Structured Quartet Research Ensemble) at the
American Institute of Mathematics in Palo Alto, California in February
2012.  It is a pleasure to thank AIM and its staff for their support,
without which our collaboration would not have been possible.  Also
our initial computations of cohomology and $K$-groups occurred on the
large shared computers \texttt{parsley} and \texttt{rosemary} at the
University of Georgia mathematics department, and we thank them for
their support in allowing these machines to be used for collaborative
mathematical research projects.

\section{The polyhedral cone}\label{sec:polyhedralcone}
Fix an imaginary quadratic field $F$ of discriminant $D<0$ with ring of
integers $\OO = \OO_D$ 
and define the element
\[\omega := 
\omega_D :=
\begin{cases}
  \sqrt{D/4} & \text{if $D \equiv 0 \bmod{4}$,}\\
  (1 + \sqrt{D})/2 & \text{if $D \equiv 1 \bmod{4}$,}
\end{cases}
\]
so that 
$F = \QQ(\omega)$ and $\OO = \ZZ[\omega]$.  Throughout we fix a complex
embedding $F \hookrightarrow \CC$ and consistently identify $F$ with its image in
$\CC$.  We also extend this identification to vectors and matrices with
coefficients in $F$.

\subsection{Hermitian forms} Let
$\cH^N(\CC)$ denote the $N^2$-dimensional real vector space of $N
\times N$ Hermitian matrices with complex coefficients.  Using the
chosen complex embedding of $F$ we can view $\cH^N(F)$, the Hermitian
matrices with coefficients in $F$, as a subset of $\cH^N(\CC)$.
Moreover this embedding allows us to view $\cH^N(\CC)$ as a
$\QQ$-vector space such that the rational points of $\cH^N(\CC)$ are
exactly $\cH^N(F)$.

Define a map
$q \colon \OO^N \to \cH^N(F)$ 
by the outer product $q(x) = x {x}^{*}$, 
where  ${*}$ denotes conjugate transpose (with conjugation being
the nontrivial complex conjugation automorphism of $F$).
Each $A \in \cH^N(\CC)$ defines a Hermitian
form $A[x]$ on $\CC^N$ by the rule
\begin{equation}
A[x] := x^* A x, \quad \text{for $x \in \CC^N$}. 
\end{equation}
Define the
non-degenerate bilinear pairing
\[\ip{\cdot, \cdot} \colon \cH^N(\CC) \times \cH^N(\CC) \to \CC\] 
by $\ip{A,B} := \Tr(AB)$.  For $x \in \OO^N$ (identified with its image
in $\CC^N$) one can easily verify that  
\begin{equation}\label{eq:traceform}
A[x] = \Tr(Aq(x)) = \ip{A, q(x)}.
\end{equation}

Let $C_N \subset \cH^N(\CC)$ denote the cone of positive definite
Hermitian matrices.

\begin{definition}
  For $A \in C_N$, we define the \emph{minimum of $A$} as 
\[m(A) := m_D(A) := \inf_{x \in \OO^N \setminus \set{0}} A[x].\]
Note that $m (A)>0$ since $A$ is positive definite.  A vector $v \in
\OO^N$ is called a \emph{minimal vector of $A$} if $A[v] = m(A)$.  
We denote the
set of minimal vectors of $A$ by $M(A)$.
\end{definition}

It should be emphasized that these notions depend on the fixed choice
of the imaginary quadratic field $F$.  Since $q(\OO^N)$ is discrete in
$\cH^N(\CC)$
and the level sets $q(x) = C$ are compact, 
the minimum for each $A$ is attained by only finitely many
minimal vectors.

From \eqref{eq:traceform}, we see that each vector $v \in \OO^N$
gives rise to a linear functional on $\cH^N(\CC)$ defined by $q(v)$. 
\begin{definition}\label{def:perfect}
We say 
a Hermitian form $A \in C_N$ is a \emph{perfect Hermitian
form over $F$} if
\[\Span_\RR\set{q(v) \mid v \in M(A)} = \cH^N(\CC).\]
\end{definition}

From Definition \ref{def:perfect} it is clear that a form is perfect
if and only if it is uniquely determined by its minimum and its
minimal vectors.  Equivalently, a form $A$ is perfect when $M(A)$
determines $A$ up to a positive real scalar.  It is convenient to
normalize a perfect form by requiring that $m (A)=1$, and we will do
so throughout this paper.  A priori there is no reason to expect that
every perfect form (up to rescaling) can be realized as 
an $F$-rational point in $\cH^{n} (\CC)$; indeed,
when one generalizes these concepts to general number fields this is
too much to expect (cf.~\cite{CubicField}).  However 
for imaginary quadratic fields this rationality property does hold:

\begin{theorem}\rftheoremcite{Theorem~3.2}{okuda-yano}
Suppose $F$ is a CM field.  Then if $A \in C_N$ is a perfect Hermitian
form over $F$ such that $m(A) = 1$, we have $A \in \cH^N(F)$.
\end{theorem}

We now construct a partial compactification of the cone $C_{N}$:
\begin{definition}
A matrix $A \in \cH^N(\CC)$ is said to have an \emph{$F$-rational
kernel} when the kernel of $A$ is spanned by vectors in $F^N \subset
\CC^N$.  Let $C^*_N \subset \cH^N (\CC)$ denote the subset of nonzero
positive semi-definite Hermitian forms with $F$-rational kernel.
\end{definition}

Let $\bG$ be the reductive group over $\QQ$ given by the restriction
of scalars $\Res_{F/\QQ}(\GL_N)$.  Thus $\bG(\QQ) = \GL_N(F)$ and
$\bG(\ZZ) = \GL_N(\OO)$.  The group $\bG(\RR) = \GL_N(\CC)$ acts on
$C^*_N$ on the left by
\[g \cdot A = gAg^*,\] 
where $g \in \GL_N(\CC)$ and $A \in C^*_N$; one can easily verify that
this action preserves $C_N$.  Let $H = R_d\bG(\RR)^0$ be the identity
component of the group of real points of the split radical of $\bG$.
Then $H \simeq \RR_+$, and as a subgroup of $\bG (\RR)$ acts on
$C^*_N$ by positive real homotheties.  Voronoi's work \cite{VoronoiI},
generalized by Ash \cite[Ch.~II]{AMRT_SmoothCompactification} and
Koecher \cite{koecher} shows that there are only finitely many perfect
Hermitian forms over $F$ modulo the action of $\GL_N(\OO)$ and $H$.

Let $X^*_N$ denote the quotient $X^*_N = H\backslash C_N^*$, and let
$\pi\colon C^*_N \to X^*_N$ denote the projection.  Then $X_N =
\pi(C_N)$ can be identified with the global Riemannian symmetric space
for the reductive group $H\backslash\GL_N(\CC)$.
\begin{remark}\label{rem:orientable}
  The action of $\bG (\RR) = \GL_{N} (\CC)$ on $\cH^N(\CC)$ gives a
representation $\rho\colon \GL_N(\CC) \to \GL_{N^2}(\RR)$.  Any
element $g$ acts on the orientation of $\cH^N(\CC)$ via the sign of
$\det(\rho(g))$.  Since $\GL_N(\CC)$ is connected, $\GL_N(\CC )$ and
its subgroup $\GL_{N} (F)$ act on $\cH^N(\CC)$ via orientation
preserving automorphisms.
\end{remark}

\subsection{Two cell complexes}\label{subsec:cell-complexes}
Let $M$ be a finite subset of $\OO^N \setminus \set{0}$.  The
\emph{perfect cone} of $M$ is the set of nonzero matrices of the form
$\sum_{v \in M} \lambda_v q(v)$, where $\lambda_v\in \RR_{\geq 0}$; by
abuse of language we also call its image by $\pi$ in $X^*_N$ a perfect
cone.  For a perfect form $A$, let $\sigma(A) \subset X^*_N$ be the
perfect cone of $M(A)$.  One can show
\cite{AMRT_SmoothCompactification, koecher} that the cells $\sigma(A)$
and their intersections, as $A$ runs over equivalence classes of
perfect forms, define a $\GL_{N} (\OO)$-invariant cell decomposition
of $X^*_N$.  In particular, for a perfect form $A \in C_N$ and an
element $\gamma \in \GL_N(\OO)$, we have
\[\gamma \cdot \sigma(A) = \text{perfect cone of $\set{\gamma v
  \mid v \in M(A)}$} = \sigma((\gamma ^*)^{-1}A \gamma ).\] 
Endow $X^*_N$ with the CW-topology \cite[Appendix]{hatcher}.

If $\tau$ is a closed cell in $X^*_N$ and $A$ is a perfect form with
$\tau \subset \sigma(A)$, we let $M(\tau)$ denote the set of vectors
$v \in M(A)$ such that $q(v) \in \tau$.  The set $M (\tau)$ is
independent of the (possible) choice of $A$.  Then $\tau$ is the image
in $X^{*}_{N}$ of the cone $C_{\tau}$ generated by $\{q (v)\mid v \in
M(\tau) \}$.  For any two closed cells $\tau$ and $\tau'$ in $X^*_N$,
we have $M(\tau) \cap M(\tau') = M(\tau \cap \tau')$.

Let $\tilde{\Sigma} \subset C^*_N$ be the (infinite) union of all
cones $C_\sigma$ such that $\sigma \in X_N^*$ has nontrivial
intersection with $X_{N}$.  One can verify that the stabilizer of
$C_\sigma$ in $\GL_{N} (\OO)$ is equal to the stabilizer of $\sigma$.
By abuse of notation, we write $M(C_\sigma)$ as $M(\sigma)$.

The collection of cones $\tilde{\Sigma }$ was used by Ash \cite{ash77,
ash-small} to construct the \emph{well-rounded retract}, a
contractible $N^2 - N$ dimensional cell complex $W$ on which
$\GL_N(\OO)$ acts cellularly with finite stabilizers of cells.  More
precisely, a nonzero finite set $M \subset
 \OO^{N}$ is called \emph{well-rounded} if the $\CC$-span
of $M$ is $\CC^N$.  For a well-rounded subset $M$, let $\sigma(M)$
denote the set of forms $A \in C_N$ with $M(A) = M$ and $m(A) = 1$.
It is easy to prove that if $\sigma(M)$ is non-empty then it is convex
and thus topologically a cell.  The well-rounded retract is then
defined to be
\[W = \bigcup_{\text{$M$ well-rounded}} \sigma(M).\]

The space $W$ is dual to the decomposition of $X_N$ in a sense made
precise in Theorem~\ref{theorem:duality} below.  For instance when $N = 2$
and $F = \QQ$, $X$ can be identified with the upper
half-plane $\hp = \{x+iy\mid y>0 \}$.  The Voronoi tessellation is the tiling of $\hp$ by
with the $\SL_{2} (\ZZ)$-translates of the ideal geodesic triangle
with vertices $\{0,1,\infty \}$.  The well-rounded retract is the dual
infinite trivalent tree (see Figure~\ref{fig:tess-spine}).

\begin{figure}
\includegraphics[scale=0.6]{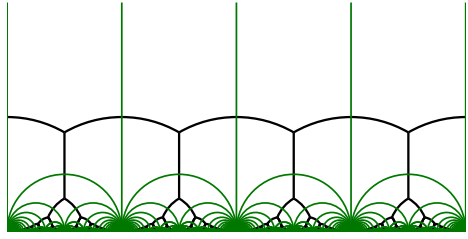}  
  \caption{Voronoi tessellation of $\hp$ shown in green, with its dual the
    well-rounded retract shown in black.} \label{fig:tess-spine}
\end{figure}

\begin{lemma}\label{lem:wellrounded-positivedefinite}
If $C_\sigma \in \tilde{\Sigma}$, then $M(\sigma)$ is well-rounded.
\end{lemma}
\begin{proof}
This is proved by McConnell in \cite[Theorem~2.11]{mcconnell-projgeom}
when $F = \QQ$, and we mimic his proof to yield the analogous result
for imaginary quadratic fields.  Let $M = M(C_\sigma)$ denote the
spanning vectors of $C_\sigma$.  Suppose $M$ is not well-rounded.
Then $M$ does not span $\CC^N$, and so there exists a non-zero vector
$w \in \CC^N$ such that $v w^* = 0$ for all $v \in M$.  Then each
$\theta \in C_\sigma$ can be written as a non-zero Hermitian form
$\theta = \sum_{v \in M}a_vq(v)$, where $a_v \geq 0$.  It follows that
\[\theta[w] = \sum_{v \in M}a_v \ip{q(v), q(w)} = \sum_{v \in M}a_v
|vw^*|^2 = 0, 
 \]
where $|\cdot|$ is the usual norm on $\CC^N$.  In particular, $\theta$
is not positive definite, contradicting the
assumption that $C_\sigma \in \tilde{\Sigma}$.
\end{proof}

\begin{theorem}\label{theorem:duality}
Let $M\subset \OO^{N}$ be well-rounded and let $\sigma(M) \in W$.
Then there is a unique cone $C_\sigma \in \tilde{\Sigma}$ such that
$M(C_\sigma) = M$.  The map $\sigma(M) \mapsto C_\sigma$ is a
canonical bijection $W \to \tilde{\Sigma}$ that is
inclusion-reversing on the face relations.
\end{theorem}

\begin{proof}
It suffices to prove that a well-rounded subset $M \subset\OO^N$ has
$\sigma(M) \neq \emptyset$ if and only if there is a cone $C_\sigma
\in \tilde{\Sigma}$ such that $M(\sigma) = M$.

Suppose $C_\sigma \in \tilde{\Sigma}$.  By
Lemma~\ref{lem:wellrounded-positivedefinite}, $M = M(C_\sigma)$ is
well-rounded.  Thus it remains to show that $\sigma(M)$ is non-empty.
We do so by constructing a form $B \in \sigma(M)$ with $M(B) =
M(C_\sigma)$.  Since $C_\sigma \in \tilde{\Sigma}$, there is a perfect
form $A$ such that $C_\sigma$ is a face of the cone $S_A =
\pi^{-1}(\sigma(A))$.  Furthermore, $C_\sigma$ can be described as the
intersection of $S_A$ with some supporting hyperplane $\set{\theta
\mid \ip{H,\theta} = 0}$ for some $H \in \cH^N(\CC)$.  It follows that
\[
\ip{H,q(v)} = 0  \quad \text{if $v \in M$} \quad \text{and} \quad 
\ip{H,q(v)} > 0  \quad \text{if $v \in M(A) \setminus M$.}
\]
Let $B = A + \rho H$.  Since $\ip{H, q(v)} = B[v]$, a standard
argument for Hermitian forms shows that for sufficiently small
positive $\rho$, $B$ is positive definite, $B[v] = m(A)$ for $v \in M$
and $B[v] > m(A)$ for $v \in M(A) \setminus M$.  Thus $M(B) = M$ and
so $B \in \sigma (M)$.  In particular, $\sigma(M)$ is non-empty.

Conversely, suppose $M$ is a well-rounded subset of $\OO^N$ with
$\sigma(M)$ non-empty.  Choose $B \in \sigma(M)$ so that $M(B) = M$.
If $B$ is perfect, then $M = M(B)$ and we are done.  Otherwise, we can
use the generalization of an algorithm of Voronoi \cite{aim-voronoi}
to find a perfect form $A$ such that $M(B) \subset M(A)$. Let $H = B -
A$.  Then
\[
\ip{H,q(v)} = 0  \quad \text{if $v \in M$} \quad \text{and} \quad 
\ip{H,q(v)} > 0  \quad \text{if $v \in M(A) \setminus M$.}
\]
Thus the hyperplane $\set{\theta \mid \ip{H,\theta} = 0}$ is a
supporting hyperplane for the subset of $S_A = \pi^{-1}(\sigma(A))$
spanned by $\set{q(v) \mid v \in M}$.  Therefore $M$ defines a face of
$S_A$ in $\tilde{\Sigma}$ as desired.
\end{proof}

\begin{remark}\label{rem:W-Sigma}
Let $\Sigma_n^*$ denote a set of representatives, modulo the action of
$\GL_N(\OO)$, of $n$-dimensional cells of $X^*_N$ that meet $X_N$.
Let $\Sigma^* = \cup_n \Sigma_n^*$.

The well-rounded retract $W$ is a proper, contractible
$\GL_N(\OO)$-complex.  Modulo $\GL_N(\OO)$, the cells in $W$ are in
bijection with cells in $\Sigma^*$, and the isomorphism classes of the
stabilizers are preserved under this bijection.  To see this, let
$\sigma(M)$ be a cell in $W$.  Then the forms $B \in \sigma(M)$ have
$M(B) = M$ and $m(B) = 1$.  Under the identification in
Theorem~\ref{theorem:duality}, $\sigma(M)$ corresponds to the cone
$C_\sigma'$ with spanning vectors $\set{q(v) \mid v \in M}$.  Let
$\sigma '$ be the corresponding cell.  There is a cell $\sigma \in
\Sigma^*$ that is $\Gamma$-equivalent to $\sigma'$.  If $\gamma \in
\GL_N(\OO)$, then $\gamma \cdot \sigma(M) =
\sigma(\set{(\gamma^*)^{-1}v \mid v \in M}) $.  Thus if $\gamma \in
\Stab(\sigma(M))$, then $(\gamma^*)^{-1} \in \Stab(C_\sigma')$.  It
is therefore clear that $\Stab(C_\sigma') = \Stab(\sigma') \simeq
\Stab(\sigma)$.  

The map $g \mapsto (g^*)^{-1}$ is an isomorphism of groups, so the
duality between $W$ mod $\Gamma$ and $\Sigma^*$ in
Theorem~\ref{theorem:duality} allows us to work with $\Sigma^*$ and its
stabilizer subgroups as if we were working with a proper, contractible
$\GL_N(\OO)$-complex.  This fact will be used later to compute the
mass formula (\S\ref{sec:mass}).
\end{remark}

\section{The cohomology and homology}\label{sec:cohomologyandhomology}
The material in this section follows \cite[\S3]{PerfFormModGrp},
\cite[\S2]{soule-3torsion}, and \cite{agm4}, all of which rely on
\cite{brown}.  Recall that $\Gamma = \GL_N(\OO)$.  In this section, we
introduce a complex $\Vor_{N,D} = (V_*(\Gamma), d_*)$ of
$\ZZ[\Gamma]$-modules whose homology is isomorphic to the group
cohomology of $\Gamma$ modulo small primes.  More precisely, for any
positive integer $n$ let $\cS_{n}$ be the Serre class of finite
abelian groups with orders only divisible by primes less than or equal
to $n$ \cite{serre}.  Then the main result of this section (Theorem
\ref{theorem:groupcoh}) is that, if $n=n (N,D)$ is larger than all the
primes dividing the orders of finite subgroups of $\Gamma$, then
modulo $\cS_{n}$ the homology of $\Vor_{N,D}$ is isomorphic to the
group cohomology of $\Gamma$.

As above let $\Sigma_n^* = \Sigma^*_n(\Gamma)$ denote a finite set of
representatives, modulo the action of $\Gamma$, of $n$-dimensional
cells of $X^*_N$ which meet $X_N$.  A cell $\sigma$ is called
orientable if every element in $\Stab(\sigma)$ preserves the
orientation of $\sigma$.  By $\Sigma_n = \Sigma_n(\Gamma)$ we denote
the set of orientable cells in $\Sigma_n^*(\Gamma)$ Let $\Sigma^* =
\cup_n \Sigma^*_n$ and let $\Sigma = \cup_n \Sigma_n$.

\subsection{Steinberg homology} 
The arithmetic group $\Gamma$---like
any arithmetic group---is a virtual duality group \cite[Theorem
11.4.4]{borel-serre}.  This means that there is a $\ZZ
[\Gamma]$-module $I$ that plays the role of an ``orientation module''
for an analogue of Poincar\'e duality: for all coefficient modules $M$
there is an isomorphism between the cohomology of $\Gamma '$ with
coefficients in $M$ and the homology of $\Gamma '$ with coefficients
in $M\otimes I$, for any torsion-free finite index subgroup $\Gamma
'\subset \Gamma $.  For more details see \cite[\S11]{borel-serre}.
The module $I$ is called the \emph{dualizing module} for $\Gamma$.
Borel and Serre prove that $I$ is isomorphic to 
$H^{\nu} (\Gamma , \ZZ [\Gamma])$, where $\nu = \vcd \Gamma$ is the
\emph{virtual cohomological dimension} of $\Gamma$
(cf.~\S\ref{sec:mass}); note that for $\Gamma = \GL_{N}(\OO)$, we have
$\nu = N^{2}-N$. 

One can show that $I$ is isomorphic to $\St_{\Gamma}\otimes \Omega $,
where $\St_{\Gamma}$ is the \emph{Steinberg module} (the top
non-vanishing reduced homology group of the Tits building attached to
$\GL_{N}/F$), and $\Omega$ is the \emph{orientation module} (which
records whether or not an element changes the orientation of the
symmetric space).  Since the action of $\GL_N(\OO)$ preserves the
orientation of $\cH^N(\CC)$, (cf.~Remark~\ref{rem:orientable}), the
module $\Omega$ is trivial, and thus $H^{\nu} (\Gamma , \ZZ [\Gamma])
= \St_\Gamma$.

For any $\Gamma$-module $M$, the \emph{Steinberg homology}
\cite[p.~279]{brown}, denoted $H^{\St}_*(\Gamma, M)$, is defined as
$H^{\St}_*(\Gamma, M) = H_*(\Gamma, \St_\Gamma \otimes M)$.  Let
$\hat{H}^*$ denote the \emph{Farrell cohomology} $\hat{H}^*(\Gamma,M)$
of $\Gamma$ with coefficients in $M$ (cf.~\cite[X.3]{brown}).  There
is a long exact sequence
\begin{equation}\label{eq:long-exact}
\cdots \to H^{\St}_{\nu - i} \to H^i \to \hat{H}^i \to  H^{\St}_{\nu
  - (i + 1)} \to H^{i + 1}\to \hat{H}^{i + 1}\to \cdots
\end{equation}
that implies that we can understand the group cohomology by
understanding the Steinberg homology and the Farrell cohomology.  In
particular, if the Farrell cohomology vanishes, (which is the case
when the torsion primes in $\Gamma$ are invertible in
$M$, see \cite[IX.9~et~seq.]{brown}), then the
Steinberg homology is exactly the group cohomology $H^{\St}_{\nu -
k}(\Gamma, M) \simeq H^k(\Gamma, M)$.

We now specialize to the case $M = \ZZ$ with trivial $\Gamma$-action,
and in the remainder of this section omit the coefficients from
homology and cohomology groups.  Note that the well-rounded retract
$W$ is an $(N^2 - N)$-dimensional proper and contractible
$\Gamma$-complex.

\begin{proposition}\label{prop:steinberg-group}
Let $b$ be an upper bound on the torsion primes for $\Gamma =
\GL_N(\OO)$.  Then modulo the Serre class $\cS_b$, we
have
\[
H^{\St}_{\nu - k}(\Gamma) \simeq H^k(\Gamma).
\]
\end{proposition}

\subsection{The Voronoi complex}\label{ss:vor}
Let $V_n(\Gamma)$ denote the free abelian group generated by
$\Sigma_n(\Gamma)$.  Let $d_n \colon V_n(\Gamma) \to V_{n -
  1}(\Gamma)$ be the map defined in \cite[\S3.1]{PerfFormModGrp}, and
  denote the complex $(V_{*} (\Gamma), d_{*})$ by $\Vor_{N,D}$.

Let $\partial X^*_N$ denote the cells in $X^*_N$ that do not meet $X$.  Then
$\partial X^*_N$ is a $\Gamma$-invariant subcomplex of $X^*_N$.  
Let $H_*^{\Gamma }(X^*_N, \partial X_N^*)$ denote the relative
equivariant homology of the
pair $(X^*_N, \partial X^*_N)$ with integral coefficients
(cf.~\cite[VII.7]{brown}).

\begin{proposition}\label{prop:Voronoi-equivariant}
  Let $b$ be an upper bound on the torsion primes for $\Gamma =
  \GL_N(\OO)$.  Modulo the Serre class $\cS_b$,  
\[
  H_n(\Vor_{N,D}) \simeq H^\Gamma_n(X_N^*, \partial X_N^*).
\]
\end{proposition}
\begin{proof}
This result follows from \cite[Proposition~2]{soule-3torsion}.  The
argument is explained in detail for $F = \QQ$ in
\cite[\S3.2]{PerfFormModGrp} and can be extended to imaginary quadratic $F$.
For the convenience of the reader we recall the argument.

There is a spectral sequence $E^r_{pq}$ 
converging to the equivariant homology groups 
\[
H^\Gamma_{p + q}(X_N^*,
\partial X_N^*)
\]
of the homology pair $(X^*_N, \partial X^*_N)$ such
that  
\begin{equation}\label{eq:brown}
E^1_{pq} = \bigoplus_{\sigma \in \Sigma^*_p}H_q(\Stab(\sigma),
\ZZ_\sigma) \Rightarrow H^\Gamma_{p + q}(X_N^*,
\partial X_N^*),
\end{equation}
where $\ZZ_\sigma$ is the orientation module for $\sigma$.  

When $\sigma$ is not orientable the homology
$H_0(\Stab(\sigma), \ZZ_\sigma)$  
is killed by $2$.
Otherwise, $H_0(\Stab(\sigma), \ZZ_\sigma) \simeq \ZZ_\sigma$.
Therefore modulo $\cS_2$ we have 
\[E^1_{n,0} \simeq \bigoplus_{\sigma \in \Sigma_n} \ZZ_\sigma.\]
 Furthermore, when $q > 0$, $H_q(\Stab(\sigma), \ZZ_\sigma)$ lies in
the Serre class $\cS_b$, where $b$ is the upper bound on the torsion
primes for $\Gamma$. In other words, modulo $\cS_b$, the $E^1$ page
of the spectral sequence \eqref{eq:brown} is concentrated in the
bottom row, and there is an identification of the bottom row groups with the
groups $V_n(\Gamma)$ defined above.  Finally,
\cite[Proposition~2]{soule-3torsion} shows that $d^1_n = d_n$.
\end{proof}
\begin{remark}
If we do not work modulo $\cS_b$, then there are other
entries in the spectral sequence to consider before we get
$H^\Gamma_n(X^*_{N}, \partial X^*_{N})$.  In particular, the small torsion
in $H_*(\Vor_{N,D})$ may not agree in general with the small torsion
in $H^\Gamma_{p + q}(X_N^*, \partial X_N^*)$.
\end{remark}

\subsection{Equivariant relative homology to Steinberg homology}
For $V \subset F^N$ a proper subspace, let $C(V)$ be the set of
matrices $A \in C_N^*$ such that the kernel of $A$ is $V
\otimes_{\QQ } \RR$, and let $X(V) = \pi(C(V))$.  The closure
$\overline{C(V)}$ of $C(V)$ in the usual topology induced from
$\cH^N(\CC)$ consists of matrices whose kernel contains $V$.  Let
$\overline{X(V)} = \pi(\overline{C(V)})$.

The following Lemma shows that $\overline{X(V)}$ is contractible.
Soul\'e proves this for $F = \QQ$ in \cite[Lemma~1]{soule-3torsion}
and \cite[Lemma~2]{soule-addendum}, and the same proof with simple
modifications applies to our setting:

\begin{lemma}\label{lem:contractible} 
  For any proper subspace $V \subset F^N$, the CW-complex
  $\overline{X(V)}$ is contractible.
\end{lemma}
\begin{proof}
Let $A$ be a perfect Hermitian form over $F$.  Then $\sigma(A) \cap
\overline{X(V)}$ is the perfect cone of the intersection $M(A) \cap
V^\perp$, where $V^\perp$ is the orthogonal complement to $V$ in
$F^N$.  It follows that $\overline{X(V)}$ is a sub-CW-complex of
$X^*_N$.

Now we show that $\overline{X(V)}$ is contractible.  Since
$\overline{X(V)}$ has the same homotopy type as $X (V)$, it is enough to prove that
$X (V)$ is contractible.  For this it suffices to prove that the
CW-topology on ${X(V)}$ coincides with its usual topology, since
$X(V)$ in this topology is clearly contractible (because it is convex).
To prove this, we argue that the covering of $X (V)$ by the closed
sets of the form $\sigma (A)\cap X (V)$, where $A$ is perfect, is
locally finite.

Given any positive definite Hermitian form $A$, let $A^{\perp}$ be its
restriction to the real space $V_{\RR}^{\perp} = V^{\perp}\otimes_{\QQ }
\RR$.  Then $X (V)$ is isomorphic to the symmetric space for the group
$\Aut (V_{\RR}^{\perp})$ via the map $A\mapsto A^{\perp}$.  If $\gamma \in
\Gamma$ satisfies $\gamma \cdot X (V)\cap X (V)\not = \emptyset$, then
in fact $\gamma $ stabilizes $V$.  Let $P\subset \Gamma$ be the
stabilizer of $V$ and let $\alpha \colon P\rightarrow \Gamma ' = \Aut
(V^{\perp}_{\RR}\cap L)$ be the projection map.  Then the 
set of cells $\sigma (A)\cap X (V)$, as $A$ ranges over the perfect
Hermitian forms, is finite modulo $\Gamma '$: if $[\sigma (A)\cap X
(V)] \cap [(\gamma \cdot \sigma (A)) \cap X (V)] \not =\emptyset$, then
$\gamma \in P$, and thus $(\gamma \cdot \sigma (A) \cap X (V)) =
\alpha (\gamma) (\sigma (A)\cap X (V))$.   

To conclude the argument one uses Siegel sets; we refer to \cite{borel}
for their definition and to \cite[Ch.~II]{AMRT_SmoothCompactification}
for their properties in our setting.  Given any point $x\in X (V)$,
one can find an open set $U\ni x$ and a Siegel set $\fS $ such that
$\fS\supset U$.  Furthermore, any cell of the form $\sigma (A)\cap X
(V)$ is itself contained in another Siegel set $\fS'$.  Thus if $U$
meets $\gamma (\sigma (A)\cap X (V))$ for $\gamma \in \Gamma '$, we
must have that $\gamma \cdot \fS '$ meets $\fS $.  But this is only
possible for finitely many $\gamma$ by the standard properties of Siegel
sets.  Thus the covering of $X (V)$ by the closed sets of the form
$\sigma (A)\cap X (V)$ ($A$ perfect) is locally finite, which
completes the proof.
\end{proof}

\begin{proposition}\label{prop:relative-steinberg}
  For every $n \geq 0$, there are canonical isomorphisms of
  $\Gamma$-modules
\[H_n(X^*_N, \partial X^*_N) \simeq 
\begin{cases}
  \St_\Gamma & \text{if $n = N - 1$,}\\
0 & \text{otherwise.}
\end{cases}
\]
\end{proposition}
\begin{proof}
Since $X^*_N$ is contractible, the long exact sequence for the pair
$(X^*_N, \partial X^*_N)$ 
\[\cdots \to H_n(\partial X^*_N) \to H_n(X^*_N) \to H_n(X^*_N, \partial
X^*_N) \to H_{n - 1}(\partial X^*_N) \to H_{n-1}(X^*_N) \to \cdots\]
implies $H_n(X^*_N, \partial X^*_N) \simeq \tilde{H}_{n-1}(\partial
X^*_N)$ for $n \geq 1$, where $\tilde{H}$ 
denotes reduced homology.

From properties of Hermitian forms, it is clear that $\overline{X(V)}
\cap \overline{X(W)} = \overline{X(V \cap W)}$ 
for each pair of proper, non-zero subspaces $V, W  \subset F^N$.  Thus the nerve
of the covering $\mathcal{U}$ of $\partial X^*_N$ by $\overline{X(V)}$
as $V$ ranges 
over non-zero, proper subspaces of $F^N$ is the spherical Tits
building $T_{F,N}$.  Since each $\overline{X(V)}$ in $\mathcal{U}$ is
contractible by Lemma~\ref{lem:contractible}, the cover $\mathcal{U}$
is a good cover (i.e.,~non-empty finite intersections are diffeomorphic
to $\RR^{d}$ for some $d$).  It follows that the relative homology groups are 
isomorphic, i.e. $\tilde{H}_n(\partial X^*_N) \simeq \tilde{H}_n(T_{F,N})$.
By the Solomon-Tits theorem \cite{solomon}, the latter is isomorphic to $\St_\Gamma$
if $n = N - 2$ and is trivial otherwise. 
\end{proof}

\begin{proposition}\label{prop:equivariant-steinberg} For all $n$, we
have
\[  H^\Gamma_{n}(X^*_N, \partial X^*_N) = H^{\St}_{n - (N - 1)}(\Gamma).\]
\end{proposition}
\begin{proof}
There is a spectral sequence \cite[equation (2)]{soule-3torsion}
computing the relative equivariant homology
\[E^2_{pq} = H_p(\Gamma, H_q(X_N^*, \partial X_N^*)) \Rightarrow
H^\Gamma_{p + q}(X^*_N, \partial X^*_N).\]
Proposition~\ref{prop:relative-steinberg} implies the $E^2$ page of
the spectral sequence is concentrated in the $q = N - 1$ column.  Then  
\[H^\Gamma_{p + (N - 1)}(X^*_N, \partial X^*_N) \simeq
 H_p(\Gamma, H_q(X^*_N, \partial X^*_n)) = H_p(\Gamma, \St_\Gamma) =
 H^{\St}_p(\Gamma),\] 
and the result follows.
\end{proof}

\begin{theorem} \label{theorem:groupcoh}
Let $b$ be an upper bound on the torsion primes for $\GL_N(\OO)$.
Modulo the Serre class $\cS_b$, 
\[H_n(\Vor_{N,D}) \simeq H^{N^2 - 1 - n}(\GL_N(\OO)).\]  
\end{theorem}
\begin{proof}
Let $\Gamma = \GL_N(\OO)$.  
Modulo $\cS_b$, Propositions~\ref{prop:Voronoi-equivariant}, 
    \ref{prop:equivariant-steinberg}, and \ref{prop:steinberg-group} imply 
\[H_n(\Vor_{N,D}) \simeq H_n^\Gamma(X^*_N, \partial X^*_N)
\simeq H^{\St}_{n - (N -  1)}(\Gamma) \simeq H^{N^2 - 1 - n}(\Gamma).\]
\end{proof}

\subsection{Torsion elements in \texorpdfstring{$\Gamma$}{Gamma}}\label{subsec:torsion}
To finish this section we discuss the possible torsion that can arise
in the stabilizer subgroups of cells in our complexes.  This allows us
to make the bound in Theorem~\ref{theorem:groupcoh} effective.

\begin{lemma}\label{lem:cyclotomic}
  Let $p$ be an odd prime, and let $F/\QQ$ be a quadratic field.  Let
  $\Phi_p = x^{p-1} + x^{p-2} + \dots + x + 1$ be the $p^\text{th}$
  cyclotomic polynomial.  Then $\Phi_p$ factors over $F$ as a product
  of irreducible polynomials of degree $(p-1)/2$ if $F =
  \QQ(\sqrt{p^*})$, where $p^* = (-1)^{(p-1)/2}p$, and is irreducible
  otherwise.   
\end{lemma}
\begin{proof}
Let $\zeta_p$ denote a primitive $p^\text{th}$ root of unity.
Consider the diagram of Galois extensions below:
\[
\begin{tikzpicture}
  \node (Fz) {$F(\zeta_p)$};
  \node (K) [below of = Fz, right of = Fz]{$\QQ(\zeta_p)$};
  \node (F) [below of = Fz, left of = Fz]{$F$};
  \node (KnF)[below of=K, left of=K] {$F \cap \QQ(\zeta_p)$};
  \node (Q) [below of=KnF]{$\QQ$};
  \draw[-] (Fz) to (K);
  \draw[-] (Fz) to (F);
  \draw[-] (KnF) to (Q);
  \draw[-] (F) to (KnF);
  \draw[-] (K) to (KnF);
\end{tikzpicture}
\]

If $F = \QQ(\sqrt{p^*})$, then $F$ is the unique quadratic subfield of
$\QQ(\zeta_p)$.  It follows that $\Phi_p$ factors over $F$ as a
product of irreducible polynomials of degree $(p-1)/2$.  If $F \neq
\QQ(\sqrt{p^*})$, then $F \cap \QQ(\zeta_p) = \QQ$.  It follows that
$[F(\zeta_p):\QQ(\zeta_p)] = [F: \QQ] =  2$ and so $[F(\zeta_p): F] =
p-1$.  Thus $\Phi_p$ is irreducible over $F$ if $F \neq
\QQ(\sqrt{p^*})$. 
\end{proof}

\begin{lemma}\label{lem:orderp}
  Let $p$ be an odd prime, and let $F/\QQ$ be a imaginary quadratic
  field.  If $g \in \GL_N(F)$ has order $p$, then 
\[p \leq 
\begin{cases}
  N + 1 & \text{if $p \equiv 1 \bmod{4}$,}\\
  2N + 1 & \text{otherwise.}
\end{cases}
\]
\end{lemma}
\begin{proof}
If $p \equiv 1 \bmod{4}$, then $p^* > 0$.  In particular, $F \neq
\QQ(\sqrt{p^*})$ and so by Lemma~\ref{lem:cyclotomic}, $\Phi_p$ is
irreducible over $F$.  Then the minimal polynomial of $g$ is
$\Phi_p$.  By the Cayley-Hamilton Theorem, $\Phi_p$ divides the
characteristic polynomial of $g$.  Therefore $p-1 \leq N$.
Similarly, if $p \not \equiv 1 \bmod{4}$, then  $(p-1)/2 \leq N$.
\end{proof}

Lemmas~\ref{lem:cyclotomic} and \ref{lem:orderp} immediately imply the following:

\begin{proposition}\label{prop:stabprimes}
If $g \in \GL_3(\OO_{F})$ has prime order $q$, then $q \in
\set{2,3,7}$ for $F = \QQ(\sqrt{-7})$ and $q\in \set{2,3}$ otherwise.
If $g \in \GL_4(\OO_{F})$ has prime order $q$, then $q \in \set{2,3,5,7}$ for
$F = \QQ(\sqrt{-7})$ and $q\in \set{2,3,5}$ otherwise. 
\end{proposition}

In Tables~\ref{tab:diff3-3}--\ref{tab:diff4-4} we give the
factorizations of the orders of the stabilizers of the cells in
$\Sigma^{*}$.  

\section{A mass formula for the Voronoi complex}\label{sec:mass} The
computation of the cell complex is a relatively difficult task and
there are many ways 
that small mistakes in the computation could cause the final answer to be incorrect.
Hence it is very important to have checks that allow us to give strong
evidence for the correctness of our computations.  One is that the
complexes we construct actually are chain complexes, namely that their
differentials square to zero.  Another is the \emph{mass formula},
stated in Theorem~\ref{theorem:mass}.  According to this formula, the
alternating sum over the cells of $\Sigma^*$ of the inverse orders of
the stabilizer subgroups must vanish.  A good reference for this
section is \cite[Ch.~IX, \S\S 6--7]{brown}, and we follow it closely.
The main theorem underlying this computation is due to Harder
\cite{harder}.

\subsection{Euler characteristics}\label{subsec:euler}
We begin by recalling some definitions.  The \emph{cohomological
dimension} $\cd\Gamma$ of a group $\Gamma$ is the largest $n\in \ZZ
\cup \{\infty \}$ such that there exists a $\ZZ \Gamma$-module $M$
with $H^{n} (\Gamma ; M)\not = 0$.  The \emph{virtual cohomological
dimension} $\vcd \Gamma$ of $\Gamma$ is defined to be the
cohomological dimension of any torsion-free finite index subgroup of
$\Gamma$ (one can show that this is well-defined).  We recall that
$\Gamma$ is said to be of \emph{finite homological type} if (i) $\vcd
\Gamma <\infty$ and (ii) for every $\ZZ \Gamma$-module $M$ that is
finitely generated as an abelian group, the homology group $H_{i}
(\Gamma ; M)$ is finitely generated for all $i$.

\begin{definition}
Let $\Gamma$ be a torsion-free group of finite homological type, and
let $H_{*} (\Gamma)$ denote the homology of $\Gamma$ with (trivial)
$\ZZ$-coefficients.  The \emph{Euler characteristic} $\chi(\Gamma)$ is
\[\chi(\Gamma) = \sum_{i}(-1)^i\rank_\ZZ(H_i(\Gamma)).\]
\end{definition}

\begin{proposition} \label{prop:euler-torsion}
\rftheoremcite{Theorem 6.3}{brown} If $\Gamma$ is torsion-free and of finite homological type and
$\Gamma' \subseteq \Gamma$ is a subgroup of finite index, then
\[\chi(\Gamma') = [\Gamma \colon \Gamma'] \cdot \chi(\Gamma).\]
\end{proposition}

One can use Proposition~\ref{prop:euler-torsion} to extend the notion of Euler
characteristic to groups with torsion.  Namely, if $\Gamma$ is an
arbitrary group of finite homological type with a torsion-free subgroup
$\Gamma'$ of finite index, one sets
\begin{equation}\label{eq:tordef}
\chi(\Gamma) = \frac{\chi(\Gamma')}{[\Gamma \colon \Gamma']}.
\end{equation}
By Proposition \ref{prop:euler-torsion}, this is independent of the
choice of $\Gamma '$.

\begin{proposition}\label{prop:euler-finite}
If $\Gamma$ is a finite group, then $\chi(\Gamma) = 1/|\Gamma|$.  
\end{proposition}
\begin{proof}
  Take $\Gamma'$ to be the trivial subgroup in \eqref{eq:tordef}.
\end{proof}

\begin{theorem} \label{theorem:euler-characteristic}
\theoremcite{harder}  Let $F$ be a number field with ring of integers $\OO$, and let
$\zeta_{F} (s)$ be the Dedekind zeta function of $F$.  Then 
\[\chi(\SL_N(\OO)) = \prod_{k = 2}^N \zeta_F(1 - k).\]  
\end{theorem}
In particular, since 
for imaginary quadratic $F$ the Dedekind zeta function
$\zeta_{F} (m)$
vanishes for 
all $m \in \mathbb{Z} < 0$, 
we have that $\chi(\SL_N(\OO)) = 0$
when $N \geq 2$.

\begin{corollary}\label{cor:euler-characteristic-gl}
Let $F$ be an imaginary quadratic field with ring of integers $\OO$.
Then $\chi(\GL_N(\OO)) = 0$ when $N \geq 2$.  
\end{corollary}
\begin{proof}
This follows immediately from Theorem~\ref{theorem:euler-characteristic}
and the definition of the Euler characteristic since $\SL_N(\OO)$ is of
finite index in $ \GL_N(\OO)$.
\end{proof}

Now we turn to a different concept, the \emph{equivariant Euler
characteristic} $\chi_{\Gamma} (X)$ of $\Gamma$. Here $X$ is any cell
complex with
$\Gamma$ action such that (i) $X$ has finitely many cells mod
$\Gamma$, and (ii) for each $\sigma\in X$, the stabilizer subgroup
$\Stab_{\Gamma } (\sigma)$ is finite.  One defines
\[
\chi_{\Gamma} (X) = \sum_{\sigma \in S} (-1)^{\Dim \sigma} 
\chi(\Stab_{\Gamma }(\sigma)),
\]
where $S$ is a set of representatives of cells of $X$ mod $\Gamma$.

\subsection{The mass formula}\label{subsec:mass}
The well-rounded retract $W$ defined in \S\ref{subsec:cell-complexes}
is a proper, contractible $\Gamma$-complex and so its equivariant 
Euler characteristic is defined.  We compute its equivariant Euler
characteristic, phrased in terms of cells in $\Sigma^*$ using
Remark~\ref{rem:W-Sigma}, to get a mass formula. 

\begin{theorem}[Mass Formula] \label{theorem:mass}
We have 
\[\sum_{\sigma \in \Sigma^*_k} (-1)^k\frac{1}{|\Stab_{\Gamma }(\sigma)|} = 0.\]  
\end{theorem}
\begin{proof}
Let $\Gamma = \GL_{N} (\OO)$.  By Proposition \ref{prop:euler-finite}
we have $\chi (\Stab_{\Gamma } (\sigma)) = 1/|\Stab_\Gamma (\sigma)|$.
Thus the result follows from Corollary
\ref{cor:euler-characteristic-gl} if we can show $\chi_{\Gamma} (W) =
\chi (\Gamma)$, where $W$ is the $\Gamma$-complex formed from the
well-rounded retract, using Theorem~\ref{theorem:duality} and
Remark~\ref{rem:W-Sigma} to identify cells of $W$ mod $\Gamma$ with
cells in $\Sigma^*$.  But according to \cite[Proposition 7.3
(e')]{brown}, this equality is true if $\Gamma$ has a torsion-free
subgroup of finite index, which is a standard fact for $\Gamma$.
\end{proof}

Using the stabilizer information in
Tables~\ref{tab:diff3-3}--\ref{tab:diff4-4} one can easily verify that
$\chi(\GL_n(\OO)) = 0$ for each of our examples.  For instance, if we
add together the terms $1/|\Stab_{\Gamma }(\sigma)|$ for cells $\sigma$ of the
same dimension to a single term for $\GL_4(\OO_{-4})$, we find
(ordering the terms by increasing dimension in $\Sigma^{*}$)
\begin{multline*}
-\frac{11}{3072} +\frac{127}{960} -\frac{4187}{2304} +\frac{28375}{2304} -\frac{868465}{18432} +\frac{126127}{1152} -\frac{81945}{512}\\ +\frac{340955}{2304} -\frac{48655}{576} +\frac{16075}{576} -\frac{21337}{4608} +\frac{101}{384} -\frac{17}{92160}  = 0.
\end{multline*}
The other groups have been checked similarly.

\section{Explicit homology classes}\label{sec:explicit}

By Theorem~\ref{theorem:groupcoh} we have $H_{N^2 - 1}(\Vor_{N,D} \otimes
\QQ) \simeq H^0(\GL_N(\OO), \QQ)$, which in turn is isomorphic to
$\QQ$.  This suggests that there should be a canonical generator for
this homology group, a fact already explored in \cite[Section
5]{PerfFormModGrp}.  An obvious choice is the analogue of the chain
presented there, namely
\[
\xi := \xi_{N,D} := \sum_\sigma \frac{1}{|\Stab(\sigma)|} [\sigma],
\]
where $\sigma$ runs through the cells in $\Sigma_{N^2 -
1}(\GL_N(\OO))$.  In this section we verify that this is true for our
examples.  We should point out that all the cells in $\Sigma_{N^2 -
1}(\GL_N(\OO))$ are orientable. The reason is that the group
$\GL_n(\CC)$ is connected and so the determinant of its action on
$\cH^N(\CC)$ is positive.  Since the stabilizers are included in
$\GL_N(\OO)$ and the faces in $\Sigma_{N^2 - 1}(\GL_N(\OO))$ are
full-dimensional the orientation has to be preserved.

\begin{theorem}\label{theorem:explicit}
When $N = 3$ and $D\geq -24$ or $N = 4$ and $D=-3,-4$, the chain $\xi$
is a cycle and thus generates $H_{N^2 - 1}(\Vor_{N,D} \otimes \QQ)$.  
\end{theorem}
\begin{proof}

The proof is an explicit computation with differential matrix $A$
representing the map $V_{N^{2}-1} (\Gamma)\rightarrow V_{N^{2}-2}
(\Gamma)$ (cf.~\S \ref{ss:vor}).  Note that the signs of the entries
of $A$ depend on a choice of orientation for each of the cells in
$\Sigma_{N^2 - 1}$ and $\Sigma_{N^2 - 2}$.  In each non-zero row of
$A$, there are exactly two non-zero entries.  Each non-zero entry
$A_{i,j}$ has absolute value $|\Stab(\sigma_j)|/|\Stab(\tau_i)|$,
where $\sigma_j \in \Sigma_{N^2 - 1}(\GL_N(\OO))$ and $\tau_i \in
\Sigma_{N^2 - 2}(\GL_N(\OO))$.  One then checks that there is a choice
of orientations such that the non-zero entries in a given row have
opposite signs.
\end{proof}

For example, consider $N = 4$ and $D = -4$.  The differential matrix is 
\[d_{15} = \begin{bmatrix}
0 & 0\\ 1920 & -256
\end{bmatrix},\] with kernel generated by $(2, 15) = 92160
({1}/{46080}, {1}/{6144})$.  The orders of the two stabilizer groups
for the cells in $\Sigma_{15}(\GL_N(\OO_{-4}))$ are $46080$ and
$6144$, respectively, and thus $\xi_{4,-4}$ is a cycle.

\begin{remark}
It seems likely that Theorem~\ref{theorem:explicit} holds for all $\GL_{N}
(\OO_{D})$, although we do not presently have a proof.
\end{remark}

\section{Results and tables}\label{sec:tables}

We conclude by presenting the results of our computations.  First we
summarize the most interesting numerical results from the tables,
namely those involving the homology computations for $\GL_*(\OO_D)$.

\begin{theorem}
Let $0>D\geq -24$ be a fundamental discriminant.
\begin{enumerate}
\item The Voronoi homology $H_n(\Vor_{3,D})$ can have non-trivial rank
only for $n=2,4,5$ and $8$.
\item In each case for $n=8$ we have $H_8(\Vor_{3,D})\cong \ZZ$.
\item In homology degree $2$, $4$ and $5$, we have the following table
  of ranks (empty slots being zero) 
\[
\begin{array}{|c|rrrrrrrrrr|}
\hline
D \phantom{\Big|} & -3 &-4 &-7 &-8 &-11 &-15 &-19 &-20 &-23 &-24  \\
\hline
\#\{\text{perf.~forms}\} \phantom{\Big|}  &2&1&2&2&12&90&157&212&870&596  \\
\hline
\rank\ H_2\phantom{\big|} & &&&&&&&1&4&1\\
\rank\ H_4\phantom{\big|} &1 &1&1&1&1&3&2&4&5&5\\
\rank\ H_5\phantom{\big|} & 1&1&2&2&2&5&3&6&1&7\\
\hline
\end{array}
  \]

\medskip
\item The $p$-Sylow subgroups for the primes occurring in (the orders of) the respective
  homology groups are given in the following table, where we use
  the shorthand $n( G)$ to denote that the $p$-Sylow subgroup in
  $H_n(\Vor_{3,D})$ is equal to $G$. 
 \def \Z {\mathbb Z}
\[
\begin{array}{|c|cccccccc|}
\hline
D \phantom{\Big|} & -3&-4&-7&-8&-11&-15&-19&-20\\ 
\hline
\Syl_2&  & 5(\Z_2)& & 4(\Z_2), 5(\Z_2)& 5(\Z_4) & 4(\Z_4) & 5(\Z_2\times \Z_8)&3(\Z_2),4(\Z_2^4), 7(\Z_2^2\times \Z_4)  \\  
\Syl_3& 7(\Z_9) &&7(\Z_3) && 7(\Z_3^2)&7(\Z_3^5)&7(\Z_3^4)&5(\Z_3), 7(\Z_3^5)\\ 
\Syl_5&  && && &7(\Z_5)&&\\
\Syl_7&  && 5(\Z_7)&& &&& 7(\Z_7)\\
\hline
\end{array}
\]

\[
\begin{array}{|c|cc|}
\hline
D \phantom{\Big|} &-23&-24 \\
\hline
\Syl_2
& 3(\Z_2),4(\Z_2^{13}),5(\Z_2^5),7(\Z_2^3\times\Z_4^2)& 3(\Z_2^7),4(\Z_2^5),5(\Z_2^8\times \Z_4^2),6(\Z_2),7(\Z_2^4\times\Z_4^2) \\
\Syl_3
& 4(\Z_3^3),7(\Z_3^7)& 7(\Z_3^3) \\
\hline
\end{array}
\]


\medskip
\medskip

\end{enumerate}
\end{theorem}

\begin{theorem}
Let $V_{4}$ be either the Voronoi complex for $\GL_{4} (\OO_{-3})$ or
$\GL_{4} (\OO_{-4})$.  Then 
 $H_n(V_{4})$ has non-trivial rank only for $n=6,\,7,\,9,\,10,\,12$, and $15$.
 More precisely,
 \begin{enumerate}
\item For $n=3k$ we have $\rank H_{3k}(V_{4})= \begin{cases}1 &\text{if }k>1,\\
0& \text{if }k=1.\end{cases}$
\item For $n=7$ we have $\rank H_{n}(V_{4})=2\,$.
\item For $n=10$ we have $\rank H_{n}(V_{4})=1\,$.
\end{enumerate}
Moreover, the only torsion primes appearing 
are 2, 3 and  5. 
\end{theorem}

Invoking Theorem~\ref{theorem:groupcoh},  we also get information on
the cohomology:

\begin{theorem} 
Let $D = -3$ or $-4$.  Then 
 $H^n(\GL_4(\OO_{D}))$ has non-trivial rank only for
$n = 0,\,3,\,5,\,6,\,8,$ and  $9$.
 More precisely,
 \begin{enumerate}
\item For $n = 3k$ we have 
$\rank H^n(\GL_4(\OO_D))= 1$ if $0 \leq k \leq 3$.
\item For $n=5$ we have $\rank H^n(\GL_4(\OO_D))=1$.
\item For $n=8$ we have $\rank H^n(\GL_4(\OO_D))=2$.
\end{enumerate}
\end{theorem}

The remainder of the text is devoted to presenting details about the
Voronoi complexes and full information about their homology.  The
notation is as follows:

\begin{itemize}
\item The first three columns concern the cell decomposition of $X_{N}^{*}$ mod $\Gamma$:
\begin{itemize}
\item $n$ is dimension of the cells in the (partially) compactified
symmetric space $X_{N}^{*}$.
\item $|\Sigma^{*}_{n}|$ is the number of $\Gamma$-orbits in the
cells that meet $X_{N}\subset X_{N}^{*}$.
\item $|\Stab |$ gives the sizes of the stabilizer
subgroups in factored form.  The notation $A (k)$ means that, of the
$|\Sigma^{*}_{n}|$ cells of dimension $n$, $k$ of them have a
stabilizer subgroup of order $A$.
\end{itemize}
\item The next four columns concern the differentials $d_{n}$ of the
Voronoi complex $\Vor_{N,D}$:
\begin{itemize}
\item $|\Sigma_{n}|$ is the number of orientable $\Gamma$-orbits in $\Sigma^{*}_{n}$.
\item $\Omega$ is the number of nonzero entries in the differential $d_{n}\colon V_{n} (\Gamma )\rightarrow V_{n-1} (\Gamma)$.
\item $\rank$ is the rank of $d_{n}$.
\item elem.~div.~gives the elementary divisors of $d_{n}$.  As in the
stabilizer column, the notation $d (k)$ means that the elementary
divisor $d$ occurs with multiplicity $k$.  If the rank of $d_{n}$
is zero then this column is left empty.
\end{itemize}
\item Finally, the last column gives the homology of the Voronoi
complex.  One can easily check that $H_{n}\simeq \ZZ^{r}\oplus
\bigoplus (\ZZ/{d}\ZZ )^{k}$, where $r=|\Sigma_{n}|-\rank
(d_{n})-\rank (d_{n+1})$ and the sum is taken over the elementary
divisors $d (k)$ from row $n+1$.  To save space, we abbreviate $\ZZ
/d\ZZ$ by $\ZZ_{d}$.  By Theorem \ref{theorem:groupcoh} we have
$H_n(\Vor_{N,D}) \simeq H^{N^2 - 1 - n}(\GL_N(\OO))$ modulo the
torsion primes in $\GL_N(\OO)$.  These primes are visible in the third
column of each table.
\end{itemize}


\newcolumntype{x}[1]{>{\RaggedRight}p{#1}}
\begin{table}[htb]
\caption{Invariants for the cell complex, differentials, and homology for $\GL_3(\OO_{-3})$.} \label{tab:diff3-3}
\begin{tabular}{|c||c|x{2in}||c|c|c|x{0.75in}||r|}  
\hline
$n$&$|\Sigma^{*}_{n}|$&$|\Stab|$&$|\Sigma_{n}|$& $\Omega$ & $\rank$ & elem.~div. & $H_{n}$ \\ 
\hline
2&1&$2^{4}3^{4}(1)$&0&0&0&
&0\\
\hline
3&2&$2^{4}3^{2}(1)$, $2^{3}3^{3}(1)$&0&0&0&
&0\\
\hline
4&3&$2^{2}3^{2}(1)$, $2^{4}3^{3}(1)$, $2^{4}3^{1}(1)$&1&0&0&
&$ \ZZ $\\
\hline
5&4&$2^{4}3^{2}(1)$, $2^{2}3^{2}(1)$, $2^{1}3^{2}(1)$, $2^{2}3^{3}(1)$&2&0&0&
&$ \ZZ $\\
\hline
6&3&$2^{2}3^{2}(1)$, $2^{1}3^{2}(1)$, $2^{2}3^{3}(1)$&1&2&1&1(1)
&0\\
\hline
7&2&$2^{4}3^{2}(1)$, $2^{2}3^{2}(1)$&1&0&0&
&$ \ZZ_{9} $\\
\hline
8&2&$2^{4}3^{4}(2)$&2&2&1&9(1)
&$ \ZZ $\\
\hline
\end{tabular}
\end{table}

\begin{table}[htb]
\caption{Invariants for the cell complex, differentials, and homology for $\GL_3(\OO_{-4})$.} \label{tab:diff3-4}
\begin{tabular}{|c||c|x{2in}||c|c|c|x{0.75in}||r|}  
\hline
$n$&$|\Sigma^{*}_{n}|$&$|\Stab|$&$|\Sigma_{n}|$& $\Omega$ & $\rank$ & elem.~div. & $H_{n}$ \\ 
\hline
2&2&$2^{7}3^{1}(2)$&0&0&0&
&0\\
\hline
3&3&$2^{3}3^{1}(1)$, $2^{5}3^{1}(2)$&0&0&0&
&0\\
\hline
4&4&$2^{4}(1)$, $2^{3}(1)$, $2^{5}(1)$, $2^{7}3^{1}(1)$&1&0&0&
&$ \ZZ $\\
\hline
5&5&$2^{2}3^{1}(2)$, $2^{3}(1)$, $2^{5}3^{1}(1)$, $2^{5}(1)$&4&0&0&
&$ \ZZ \oplus \ZZ_{2}$\\
\hline
6&3&$2^{2}3^{1}(2)$, $2^{4}(1)$&3&10&3&1(2), 2(1)
&0\\
\hline
7&1&$2^{4}(1)$&0&0&0&
&0\\
\hline
8&1&$2^{7}3^{1}(1)$&1&0&0&
&$ \ZZ $\\
\hline
\end{tabular}
\end{table}

\begin{table}[htb]
\caption{Invariants for the cell complex, differentials, and homology for $\GL_3(\OO_{-7})$.} \label{tab:diff3-7}
\begin{tabular}{|c||c|x{2in}||c|c|c|x{0.75in}||r|}  
\hline
$n$&$|\Sigma^{*}_{n}|$&$|\Stab|$&$|\Sigma_{n}|$& $\Omega$ & $\rank$ & elem.~div. & $H_{n}$ \\ 
\hline
2&3&$2^{4}3^{1}(3)$&0&0&0&
&0\\
\hline
3&6&$2^{2}3^{1}(2)$, $2^{3}3^{1}(1)$, $2^{2}(1)$, $2^{4}3^{1}(1)$, $2^{4}(1)$&0&0&0&
&0\\
\hline
4&9&$2^{2}3^{1}(1)$, $2^{1}(1)$, $2^{3}3^{1}(1)$, $2^{2}(2)$, $2^{4}(1)$, $2^{3}(3)$&3&0&0&
&$ \ZZ $\\
\hline
5&11&$2^{2}3^{1}(1)$, $2^{1}(2)$, $2^{2}(1)$, $2^{4}3^{1}(3)$, $2^{1}3^{1}(3)$, $2^{3}(1)$&10&8&2&1(2)
&$ \ZZ^{2} \oplus \ZZ_{7}$\\
\hline
6&8&$2^{2}(2)$, $2^{1}3^{1}(2)$, $2^{3}(3)$, $2^{1}3^{1}7^{1}(1)$&6&19&6&1(5), 7(1)
&0\\
\hline
7&2&$2^{1}7^{1}(1)$, $2^{2}(1)$&1&0&0&
&$ \ZZ_{3} $\\
\hline
8&2&$2^{4}3^{1}7^{1}(1)$, $2^{1}3^{1}7^{1}(1)$&2&2&1&3(1)
&$ \ZZ $\\
\hline
\end{tabular}
\end{table}

\begin{table}[htb]
\caption{Invariants for the cell complex, differentials, and homology for $\GL_3(\OO_{-8})$.} \label{tab:diff3-8}
\begin{tabular}{|c||c|x{2in}||c|c|c|x{0.75in}||r|}  
\hline
$n$&$|\Sigma^{*}_{n}|$&$|\Stab|$&$|\Sigma_{n}|$& $\Omega$ & $\rank$ & elem.~div. & $H_{n}$ \\ 
\hline
2&5&$2^{2}3^{1}(2)$, $2^{4}3^{1}(2)$, $2^{4}(1)$&0&0&0&
&0\\
\hline
3&16&$2^{1}(2)$, $2^{2}(5)$, $2^{3}3^{1}(1)$, $2^{2}3^{1}(5)$, $2^{4}3^{1}(1)$, $2^{5}(2)$&2&0&0&
&0\\
\hline
4&26&$2^{1}(14)$, $2^{2}(9)$, $2^{3}(1)$, $2^{5}3^{1}(1)$, $2^{4}(1)$&16&12&2&1(2)
&$ \ZZ \oplus \ZZ_{2}$\\
\hline
5&37&$2^{1}(25)$, $2^{2}(4)$, $2^{1}3^{1}(4)$, $2^{3}(2)$, $2^{4}3^{1}(2)$&36&104&13&1(12), 2(1)
&$ \ZZ^{2} \oplus \ZZ_{2}$\\
\hline
6&28&$2^{1}(18)$, $2^{2}(2)$, $2^{1}3^{1}(4)$, $2^{3}(2)$, $2^{2}3^{1}(2)$&26&166&21&1(20), 2(1)
&0\\
\hline
7&7&$2^{1}(6)$, $2^{2}(1)$&6&45&5&1(5)
&0\\
\hline
8&2&$2^{1}3^{1}(1)$, $2^{5}(1)$&2&6&1&1(1)
&$ \ZZ $\\
\hline
\end{tabular}
\end{table}

\begin{table}[htb]
\caption{Invariants for the cell complex, differentials, and homology for $\GL_3(\OO_{-11})$.} \label{tab:diff3-11}
\begin{tabular}{|c||c|x{2in}||c|c|c|x{0.75in}||r|}  
\hline
$n$&$|\Sigma^{*}_{n}|$&$|\Stab|$&$|\Sigma_{n}|$& $\Omega$ & $\rank$ & elem.~div. & $H_{n}$ \\ 
\hline
2&8&$2^{2}3^{1}(2)$, $2^{3}3^{1}(1)$, $2^{2}(2)$, $2^{4}3^{1}(1)$, $2^{4}(2)$&1&0&0&
&0\\
\hline
3&34&$2^{1}(13)$, $2^{2}(11)$, $2^{3}(2)$, $2^{4}3^{1}(1)$, $2^{2}3^{1}(4)$, $2^{3}3^{1}(3)$&15&6&1&1(1)
&0\\
\hline
4&91&$2^{1}(67)$, $2^{2}(15)$, $2^{1}3^{1}(1)$, $2^{3}(5)$, $2^{2}3^{1}(1)$, $2^{4}(1)$, $2^{4}3^{1}(1)$&75&193&14&1(14)
&$ \ZZ $\\
\hline
5&150&$2^{1}(124)$, $2^{2}(13)$, $2^{1}3^{1}(7)$, $2^{3}(2)$, $2^{2}3^{1}(3)$, $2^{4}3^{1}(1)$&147&700&60&1(60)
&$ \ZZ^{2} \oplus \ZZ_{4}$\\
\hline
6&125&$2^{1}(110)$, $2^{2}(6)$, $2^{1}3^{1}(6)$, $2^{3}(3)$&122&859&85&1(84), 4(1)
&0\\
\hline
7&51&$2^{1}(44)$, $2^{2}(2)$, $2^{1}3^{1}(3)$, $2^{2}3^{1}(1)$, $2^{4}(1)$&48&404&37&1(37)
&$ \ZZ_{3}^{2} $\\
\hline
8&12&$2^{1}(4)$, $2^{1}3^{1}(6)$, $2^{4}(2)$&12&88&11&1(9), 3(2)
&$ \ZZ $\\
\hline
\end{tabular}
\end{table}

\begin{table}[htb]
\caption{Invariants for the cell complex, differentials, and homology for $\GL_3(\OO_{-15})$.} \label{tab:diff3-15}
\begin{tabular}{|c||c|x{2in}||c|c|c|x{0.75in}||r|}  
\hline
$n$&$|\Sigma^{*}_{n}|$&$|\Stab|$&$|\Sigma_{n}|$& $\Omega$ & $\rank$ & elem.~div. & $H_{n}$ \\ 
\hline
2&34&$2^{1}(4)$, $2^{2}(8)$, $2^{1}3^{1}(1)$, $2^{3}(8)$, $2^{4}3^{1}(3)$, $2^{2}3^{1}(3)$, $2^{4}(7)$&10&0&0&
&$ \ZZ $\\
\hline
3&217&$2^{1}(102)$, $2^{2}(77)$, $2^{1}3^{1}(2)$, $2^{3}(23)$, $2^{2}3^{1}(5)$, $2^{4}(2)$, $2^{3}3^{1}(3)$, $2^{4}3^{1}(3)$&128&175&9&1(9)
&0\\
\hline
4&689&$2^{1}(546)$, $2^{2}(114)$, $2^{1}3^{1}(1)$, $2^{3}(20)$, $2^{2}3^{1}(2)$, $2^{4}(5)$, $2^{3}3^{1}(1)$&604&2112&119&1(119)
&$ \ZZ^{3} \oplus \ZZ_{2}^{4}$\\
\hline
5&1224&$2^{1}(1109)$, $2^{2}(84)$, $2^{1}3^{1}(7)$, $2^{3}(13)$, $2^{2}3^{1}(5)$, $2^{4}(3)$, $2^{4}3^{1}(3)$&1185&6373&482&1(478), 2(4)
&$ \ZZ^{5} $\\
\hline
6&1139&$2^{1}(1081)$, $2^{2}(47)$, $2^{1}3^{1}(7)$, $2^{3}(4)$&1102&7771&698&1(698)
&0\\
\hline
7&522&$2^{1}(489)$, $2^{2}(30)$, $2^{1}3^{1}(1)$, $2^{3}(1)$, $2^{2}3^{1}(1)$&493&4162&404&1(404)
&$ \ZZ_{3}^{2}\oplus\ZZ_{6}^{2}\oplus\ZZ_{12} $\\
\hline
8&90&$2^{1}(78)$, $2^{2}(3)$, $2^{1}3^{1}(5)$, $2^{3}(2)$, $2^{2}3^{1}(2)$&90&972&89&1(84), 3(2), 6(2), 12(1)
&$ \ZZ $\\
\hline
\end{tabular}
\end{table}

\begin{table}[htb]
\caption{Invariants for the cell complex, differentials, and homology for $\GL_3(\OO_{-19})$.} \label{tab:diff3-19}
\begin{tabular}{|c||c|x{2in}||c|c|c|x{0.75in}||r|}  
\hline
$n$&$|\Sigma^{*}_{n}|$&$|\Stab|$&$|\Sigma_{n}|$& $\Omega$ & $\rank$ & elem.~div. & $H_{n}$ \\ 
\hline
2&43&$2^{1}(23)$, $2^{2}(10)$, $2^{4}3^{1}(1)$, $2^{1}3^{1}(3)$, $2^{3}(1)$, $2^{2}3^{1}(2)$, $2^{4}(2)$, $2^{3}3^{1}(1)$&29&0&0&
&0\\
\hline
3&359&$2^{1}(304)$, $2^{2}(38)$, $2^{1}3^{1}(3)$, $2^{3}(7)$, $2^{2}3^{1}(4)$, $2^{3}3^{1}(2)$, $2^{4}3^{1}(1)$&314&664&29&1(29)
&0\\
\hline
4&1293&$2^{1}(1234)$, $2^{2}(52)$, $2^{3}(4)$, $2^{4}(1)$, $2^{3}3^{1}(1)$, $2^{4}3^{1}(1)$&1255&5410&285&1(285)
&$ \ZZ^{2} $\\
\hline
5&2347&$2^{1}(2287)$, $2^{2}(37)$, $2^{1}3^{1}(13)$, $2^{3}(7)$, $2^{2}3^{1}(2)$, $2^{4}3^{1}(1)$&2339&13596&968&1(968)
&$ \ZZ^{3} \oplus \ZZ_{2}\oplus\ZZ_{8}$\\
\hline
6&2169&$2^{1}(2137)$, $2^{2}(15)$, $2^{1}3^{1}(10)$, $2^{3}(7)$&2164&15404&1368&1(1366), 2(1), 8(1)
&0\\
\hline
7&958&$2^{1}(950)$, $2^{2}(6)$, $2^{1}3^{1}(1)$, $2^{4}(1)$&952&8181&796&1(796)
&$ \ZZ_{3}^{4} $\\
\hline
8&157&$2^{1}(149)$, $2^{1}3^{1}(6)$, $2^{4}(2)$&157&1884&156&1(152), 3(4)
&$ \ZZ $\\
\hline
\end{tabular}
\end{table}

\begin{table}[htb]
\caption{Invariants for the cell complex, differentials, and homology for $\GL_3(\OO_{-20})$.} \label{tab:diff3-20}
\begin{tabular}{|c||c|x{2in}||c|c|c|x{0.75in}||r|}  
\hline
$n$&$|\Sigma^{*}_{n}|$&$|\Stab|$&$|\Sigma_{n}|$& $\Omega$ & $\rank$ & elem.~div. & $H_{n}$ \\ 
\hline
2&69&$2^{1}(21)$, $2^{2}(26)$, $2^{1}3^{1}(4)$, $2^{3}(6)$, $2^{2}3^{1}(3)$, $2^{4}(7)$, $2^{4}3^{1}(2)$&31&0&0&
&$ \ZZ $\\
\hline
3&538&$2^{1}(398)$, $2^{2}(98)$, $2^{1}3^{1}(4)$, $2^{3}(22)$, $2^{2}3^{1}(7)$, $2^{4}(3)$, $2^{3}3^{1}(4)$, $2^{4}3^{1}(2)$&425&772&30&1(30)
&$ \ZZ_{2} $\\
\hline
4&1895&$2^{1}(1721)$, $2^{2}(153)$, $2^{3}(15)$, $2^{2}3^{1}(1)$, $2^{4}(4)$, $2^{4}3^{1}(1)$&1804&7464&395&1(394), 2(1)
&$ \ZZ^{4} \oplus \ZZ_{2}^{4}$\\
\hline
5&3382&$2^{1}(3223)$, $2^{2}(117)$, $2^{1}3^{1}(15)$, $2^{3}(15)$, $2^{2}3^{1}(7)$, $2^{4}(4)$, $2^{4}3^{1}(1)$&3345&19167&1405&1(1401), 2(4)
&$ \ZZ^{6} \oplus \ZZ_{3}$\\
\hline
6&3061&$2^{1}(2976)$, $2^{2}(61)$, $2^{1}3^{1}(15)$, $2^{3}(8)$, $2^{2}3^{1}(1)$&3017&21502&1934&1(1933), 3(1)
&0\\
\hline
7&1330&$2^{1}(1293)$, $2^{2}(37)$&1294&11127&1083&1(1083)
&$ \ZZ_{3}^{2}\oplus\ZZ_{6}^{2}\oplus\ZZ_{12} $\\
\hline
8&212&$2^{1}(202)$, $2^{2}(2)$, $2^{1}3^{1}(5)$, $2^{3}(2)$, $2^{2}3^{1}(1)$&212&2532&211&1(206), 3(2), 6(2), 12(1)
&$ \ZZ $\\
\hline
\end{tabular}
\end{table}

\begin{table}[htb]
\caption{Invariants for the cell complex, differentials, and homology for $\GL_3(\OO_{-23})$.} \label{tab:diff3-23}
\begin{tabular}{|c||c|x{2in}||c|c|c|x{0.75in}||r|}  
\hline
$n$&$|\Sigma^{*}_{n}|$&$|\Stab|$&$|\Sigma_{n}|$& $\Omega$ & $\rank$ & elem.~div. & $H_{n}$ \\ 
\hline
2&204&$2^{1}(89)$, $2^{2}(65)$, $2^{1}3^{1}(4)$, $2^{3}(27)$, $2^{2}3^{1}(6)$, $2^{4}(10)$, $2^{4}3^{1}(3)$&126&0&0&
&$ \ZZ^{4} $\\
\hline
3&1777&$2^{1}(1402)$, $2^{2}(295)$, $2^{1}3^{1}(2)$, $2^{3}(56)$, $2^{2}3^{1}(11)$, $2^{4}(3)$, $2^{3}3^{1}(5)$, $2^{4}3^{1}(3)$&1477&3272&122&1(122)
&$ \ZZ_{2} $\\
\hline
4&6589&$2^{1}(6112)$, $2^{2}(434)$, $2^{3}(35)$, $2^{2}3^{1}(2)$, $2^{4}(5)$, $2^{3}3^{1}(1)$&6285&26837&1355&1(1354), 2(1)
&$ \ZZ^{5} \oplus \ZZ_{2}^{10}\oplus\ZZ_{6}^{3}$\\
\hline
5&12214&$2^{1}(11866)$, $2^{2}(291)$, $2^{1}3^{1}(19)$, $2^{3}(23)$, $2^{2}3^{1}(6)$, $2^{4}(6)$, $2^{4}3^{1}(3)$&12119&69891&4925&1(4912), 2(10), 6(3)
&$ \ZZ^{10} \oplus \ZZ_{2}^{5}$\\
\hline
6&11627&$2^{1}(11461)$, $2^{2}(138)$, $2^{1}3^{1}(16)$, $2^{3}(10)$, $2^{2}3^{1}(2)$&11568&81720&7184&1(7179), 2(5)
&0\\
\hline
7&5303&$2^{1}(5250)$, $2^{2}(48)$, $2^{1}3^{1}(2)$, $2^{2}3^{1}(2)$, $2^{4}(1)$&5253&44741&4384&1(4384)
&$ \ZZ_{12}^{2}\oplus\ZZ_{3}^{2}\oplus\ZZ_{6}^{3} $\\
\hline
8&870&$2^{1}(853)$, $2^{2}(3)$, $2^{1}3^{1}(10)$, $2^{3}(2)$, $2^{4}(2)$&870&10464&869&1(862), 12(2), 3(2), 6(3)
&$ \ZZ $\\
\hline
\end{tabular}
\end{table}

\begin{table}[htb]
\caption{Invariants for the cell complex, differentials, and homology for $\GL_3(\OO_{-24})$.} \label{tab:diff3-24}
\begin{tabular}{|c||c|x{2in}||c|c|c|x{0.75in}||r|}  
\hline
$n$&$|\Sigma^{*}_{n}|$&$|\Stab|$&$|\Sigma_{n}|$& $\Omega$ & $\rank$ & elem.~div. & $H_{n}$ \\ 
\hline
2&158&$2^{1}(90)$, $2^{2}(41)$, $2^{1}3^{1}(3)$, $2^{3}(13)$, $2^{2}3^{1}(1)$, $2^{4}(8)$, $2^{4}3^{1}(2)$&104&0&0&
&$ \ZZ $\\
\hline
3&1396&$2^{1}(1214)$, $2^{2}(142)$, $2^{1}3^{1}(4)$, $2^{3}(24)$, $2^{2}3^{1}(6)$, $2^{4}(1)$, $2^{3}3^{1}(3)$, $2^{4}3^{1}(2)$&1247&2967&103&1(103)
&$ \ZZ_{2}^{7} $\\
\hline
4&5090&$2^{1}(4859)$, $2^{2}(199)$, $2^{1}3^{1}(1)$, $2^{3}(23)$, $2^{2}3^{1}(1)$, $2^{4}(5)$, $2^{3}3^{1}(1)$, $2^{4}3^{1}(1)$&4957&22280&1144&1(1137), 2(7)
&$ \ZZ^{5} \oplus \ZZ_{2}^{5}$\\
\hline
5&9091&$2^{1}(8889)$, $2^{2}(161)$, $2^{1}3^{1}(12)$, $2^{3}(20)$, $2^{2}3^{1}(5)$, $2^{4}(3)$, $2^{4}3^{1}(1)$&9043&53385&3808&1(3803), 2(5)
&$ \ZZ^{7} \oplus \ZZ_{2}^{8}\oplus\ZZ_{4}^{2}$\\
\hline
6&8319&$2^{1}(8187)$, $2^{2}(102)$, $2^{1}3^{1}(13)$, $2^{3}(15)$, $2^{2}3^{1}(2)$&8263&58948&5228&1(5218), 2(8), 4(2)
&$ \ZZ_{2} $\\
\hline
7&3662&$2^{1}(3617)$, $2^{2}(42)$, $2^{3}(2)$, $2^{4}(1)$&3630&31020&3035&1(3034), 2(1)
&$ \ZZ_{12}^{2}\oplus\ZZ_{2}^{3}\oplus\ZZ_{6} $\\
\hline
8&596&$2^{1}(578)$, $2^{2}(8)$, $2^{1}3^{1}(2)$, $2^{3}(4)$, $2^{2}3^{1}(2)$, $2^{4}(2)$&596&7188&595&1(589), 12(2), 2(3), 6(1)
&$ \ZZ $\\
\hline
\end{tabular}
\end{table}

\begin{table}[htb]
\caption{Invariants for the cell complex, differentials, and homology for $\GL_4(\OO_{-3})$.} \label{tab:diff4-3}
\begin{tabular}{|c||c|x{2in}||c|c|c|x{0.75in}||r|}  
\hline
$n$&$|\Sigma^{*}_{n}|$&$|\Stab|$&$|\Sigma_{n}|$& $\Omega$ & $\rank$ & elem.~div. & $H_{n}$ \\ 
\hline
3&2&$2^{4}3^{5}(1)$, $2^{7}3^{5}(1)$&0&0&0&
&0\\
\hline
4&5&$2^{4}3^{2}(1)$, $2^{4}3^{2}5^{1}(1)$, $2^{5}3^{3}(1)$, $2^{5}3^{4}(1)$, $2^{3}3^{4}(1)$&0&0&0&
&0\\
\hline
5&12&$2^{2}3^{2}(2)$, $2^{3}3^{2}(1)$, $2^{2}3^{3}(1)$, $2^{3}3^{1}(1)$, $2^{5}3^{4}(1)$, $2^{3}3^{3}(1)$, $2^{4}3^{1}(2)$, $2^{5}3^{2}(2)$, $2^{6}3^{4}(1)$&0&0&0&
&0\\
\hline
6&34&$2^{5}3^{4}(1)$, $2^{1}3^{1}(1)$, $2^{4}3^{1}(2)$, $2^{2}3^{1}(8)$, $2^{5}3^{3}(1)$, $2^{1}3^{2}(2)$, $2^{3}3^{4}(1)$, $2^{3}3^{1}(4)$, $2^{5}3^{2}(1)$, $2^{3}3^{3}(1)$, $2^{4}3^{2}(2)$, $2^{2}3^{3}(4)$, $2^{2}3^{2}(6)$&8&0&0&
&$ \ZZ $\\
\hline
7&82&$2^{1}3^{1}(21)$, $2^{5}3^{1}(1)$, $2^{2}3^{1}(23)$, $2^{1}3^{2}(7)$, $2^{2}3^{3}(3)$, $2^{3}3^{1}(5)$, $2^{3}3^{4}(1)$, $2^{2}3^{2}(8)$, $2^{3}3^{3}(2)$, $2^{7}3^{4}(1)$, $2^{4}3^{1}(4)$, $2^{1}3^{3}(2)$, $2^{4}3^{2}(1)$, $2^{3}3^{2}(3)$&50&58&7&1(7)
&$ \ZZ^{2} \oplus \ZZ_{2}$\\
\hline
8&166&$2^{3}3^{3}(2)$, $2^{1}3^{1}(88)$, $2^{4}3^{3}(1)$, $2^{2}3^{1}(36)$, $2^{1}3^{2}(13)$, $2^{5}3^{3}(1)$, $2^{3}3^{1}(5)$, $2^{2}3^{2}(7)$, $2^{4}3^{1}(2)$, $2^{1}3^{3}(2)$, $2^{3}3^{2}(5)$, $2^{2}3^{3}(3)$, $2^{2}3^{5}(1)$&129&604&41&1(40), 2(1)
&$ \ZZ_{9} $\\
\hline
9&277&$2^{1}3^{1}(191)$, $2^{5}3^{2}(1)$, $2^{2}3^{1}(34)$, $2^{1}3^{2}(17)$, $2^{5}3^{5}(2)$, $2^{3}3^{1}(6)$, $2^{1}3^{1}5^{1}(1)$, $2^{2}3^{2}(9)$, $2^{1}3^{3}(2)$, $2^{3}3^{2}(3)$, $2^{5}3^{1}(1)$, $2^{2}3^{3}(7)$, $2^{7}3^{3}(1)$, $2^{2}3^{5}(1)$, $2^{4}3^{2}5^{1}(1)$&228&1616&88&1(87), 9(1)
&$ \ZZ \oplus \ZZ_{24}$\\
\hline
10&324&$2^{1}3^{1}(246)$, $2^{2}3^{1}(35)$, $2^{1}3^{2}(16)$, $2^{3}3^{1}(7)$, $2^{2}3^{2}(9)$, $2^{1}3^{3}(2)$, $2^{4}3^{3}(1)$, $2^{2}3^{3}(5)$, $2^{3}3^{3}(2)$, $2^{5}3^{2}(1)$&286&2531&139&1(138), 24(1)
&$ \ZZ \oplus \ZZ_{2}^{4}$\\
\hline
11&259&$2^{1}3^{1}(200)$, $2^{2}3^{1}(24)$, $2^{1}3^{2}(11)$, $2^{3}3^{1}(6)$, $2^{2}3^{2}(9)$, $2^{2}3^{4}(1)$, $2^{4}3^{1}(1)$, $2^{1}3^{3}(1)$, $2^{2}3^{3}(2)$, $2^{4}3^{2}(2)$, $2^{4}3^{3}(2)$&237&2283&146&1(142), 2(4)
&$ \ZZ_{2}^{2}\oplus\ZZ_{6} $\\
\hline
12&142&$2^{1}3^{1}(91)$, $2^{2}3^{1}(20)$, $2^{1}3^{2}(9)$, $2^{3}3^{1}(5)$, $2^{2}3^{2}(11)$, $2^{1}3^{3}(4)$, $2^{2}3^{3}(1)$, $2^{4}3^{2}(1)$&122&1252&91&1(88), 2(2), 6(1)
&$ \ZZ \oplus \ZZ_{3}\oplus\ZZ_{12}$\\
\hline
13&48&$2^{1}3^{3}(1)$, $2^{3}3^{3}(1)$, $2^{1}3^{1}(22)$, $2^{2}3^{1}(6)$, $2^{1}3^{2}(7)$, $2^{3}3^{1}(2)$, $2^{2}3^{2}(8)$, $2^{4}3^{1}(1)$&36&369&30&1(28), 3(1), 12(1)
&$ \ZZ_{15} $\\
\hline
14&15&$2^{1}3^{2}(2)$, $2^{2}3^{2}(2)$, $2^{1}3^{3}(1)$, $2^{2}3^{3}(3)$, $2^{1}3^{1}(1)$, $2^{4}3^{2}5^{1}(1)$, $2^{3}3^{1}(1)$, $2^{2}3^{1}(2)$, $2^{1}3^{1}5^{1}(1)$, $2^{6}3^{2}(1)$&10&51&6&1(5), 15(1)
&$ \ZZ_{12}\oplus\ZZ_{288} $\\
\hline
15&5&$2^{6}3^{3}(1)$, $2^{7}3^{5}5^{1}(1)$, $2^{4}3^{5}(1)$, $2^{2}3^{3}(2)$&5&16&4&1(2), 12(1), 288(1)
&$ \ZZ $\\
\hline
\end{tabular}
\end{table}

\begin{table}[htb]
\caption{Invariants for the cell complex, differentials, and homology for $\GL_4(\OO_{-4})$.} \label{tab:diff4-4}
\begin{tabular}{|c||c|x{2in}||c|c|c|x{0.75in}||r|}  
\hline
$n$&$|\Sigma^{*}_{n}|$&$|\Stab|$&$|\Sigma_{n}|$& $\Omega$ & $\rank$ & elem.~div. & $H_{n}$ \\ 
\hline
3&4&$2^{11}3^{1}(2)$, $2^{9}3^{1}(1)$, $2^{7}3^{1}(1)$&0&0&0&
&0\\
\hline
4&10&$2^{3}3^{1}(1)$, $2^{8}(1)$, $2^{4}3^{1}(1)$, $2^{5}3^{1}5^{1}(1)$, $2^{7}(1)$, $2^{5}3^{1}(2)$, $2^{8}3^{1}(1)$, $2^{7}3^{1}(1)$, $2^{5}(1)$&0&0&0&
&0\\
\hline
5&33&$2^{10}3^{1}(1)$, $2^{2}(1)$, $2^{7}3^{2}(1)$, $2^{6}3^{1}(1)$, $2^{3}(6)$, $2^{4}3^{1}(2)$, $2^{4}(7)$, $2^{7}(1)$, $2^{5}3^{1}(2)$, $2^{3}3^{1}(2)$, $2^{10}(1)$, $2^{6}(2)$, $2^{5}(6)$&5&0&0&
&0\\
\hline
6&98&$2^{7}(1)$, $2^{7}3^{1}(2)$, $2^{2}(26)$, $2^{3}(37)$, $2^{2}3^{1}(1)$, $2^{4}(10)$, $2^{8}3^{2}(1)$, $2^{3}3^{1}(1)$, $2^{5}(9)$, $2^{4}3^{1}(4)$, $2^{6}(1)$, $2^{6}3^{1}(1)$, $2^{5}3^{1}(4)$&48&35&5&1(5)
&$ \ZZ \oplus \ZZ_{2}^{4}$\\
\hline
7&258&$2^{2}(147)$, $2^{3}(69)$, $2^{2}3^{1}(1)$, $2^{4}(18)$, $2^{3}3^{1}(2)$, $2^{5}(10)$, $2^{4}3^{1}(3)$, $2^{6}(2)$, $2^{5}3^{1}(2)$, $2^{7}(3)$, $2^{11}3^{2}(1)$&189&682&42&1(38), 2(4)
&$ \ZZ^{2} \oplus \ZZ_{2}^{3}$\\
\hline
8&501&$2^{2}(397)$, $2^{3}(67)$, $2^{2}3^{1}(4)$, $2^{4}(18)$, $2^{3}3^{1}(4)$, $2^{5}(5)$, $2^{4}3^{1}(1)$, $2^{6}(3)$, $2^{7}(1)$, $2^{5}3^{2}(1)$&435&2972&145&1(142), 2(3)
&$ \ZZ_{2}^{2}\oplus\ZZ_{4} $\\
\hline
9&704&$2^{2}(603)$, $2^{3}(58)$, $2^{2}3^{1}(6)$, $2^{4}(15)$, $2^{2}5^{1}(1)$, $2^{3}3^{1}(7)$, $2^{5}(6)$, $2^{4}3^{1}(2)$, $2^{6}(1)$, $2^{3}3^{2}(1)$, $2^{7}(1)$, $2^{5}3^{1}5^{1}(1)$, $2^{8}3^{2}(1)$, $2^{9}3^{1}(1)$&639&5928&290&1(287), 2(2), 4(1)
&$ \ZZ \oplus \ZZ_{2}^{4}\oplus\ZZ_{4}$\\
\hline
10&628&$2^{2}(571)$, $2^{3}(31)$, $2^{2}3^{1}(4)$, $2^{4}(13)$, $2^{3}3^{1}(3)$, $2^{5}(1)$, $2^{6}(2)$, $2^{3}3^{2}(1)$, $2^{7}(1)$, $2^{8}(1)$&597&6701&348&1(343), 2(4), 4(1)
&$ \ZZ \oplus \ZZ_{2}^{7}\oplus\ZZ_{4}^{2}\oplus\ZZ_{8}\oplus\ZZ_{24}$\\
\hline
11&369&$2^{2}(320)$, $2^{3}(25)$, $2^{2}3^{1}(4)$, $2^{4}(12)$, $2^{3}3^{1}(4)$, $2^{5}(3)$, $2^{6}3^{2}(1)$&346&4544&248&1(237), 2(7), 4(2), 8(1), 24(1)
&0\\
\hline
12&130&$2^{2}(103)$, $2^{3}(9)$, $2^{2}3^{1}(8)$, $2^{4}(3)$, $2^{3}3^{1}(1)$, $2^{5}(2)$, $2^{2}3^{2}(1)$, $2^{4}3^{1}(1)$, $2^{6}(1)$, $2^{5}3^{1}(1)$&120&1787&98&1(98)
&$ \ZZ \oplus \ZZ_{2}\oplus\ZZ_{8}$\\
\hline
13&31&$2^{2}(13)$, $2^{6}3^{1}(1)$, $2^{3}(7)$, $2^{2}3^{1}(4)$, $2^{4}(1)$, $2^{3}3^{1}(2)$, $2^{6}(1)$, $2^{5}3^{2}(1)$, $2^{9}(1)$&22&337&21&1(19), 2(1), 8(1)
&$ \ZZ_{5} $\\
\hline
14&7&$2^{3}3^{1}(2)$, $2^{4}(2)$, $2^{5}3^{1}5^{1}(1)$, $2^{2}5^{1}(1)$, $2^{7}3^{1}(1)$&2&3&1&5(1)
&$ \ZZ_{128} $\\
\hline
15&2&$2^{10}3^{2}5^{1}(1)$, $2^{11}3^{1}(1)$&2&2&1&128(1)
&$ \ZZ $\\
\hline
\end{tabular}
\end{table}

\clearpage

\bibliographystyle{amsplain_initials_eprint}
\bibliography{HomolArith}

\end{document}